\documentclass{amsart}

\newtheorem{theorem}{Theorem}[section]
\newtheorem{lemma}[theorem]{Lemma}

\theoremstyle{definition}
\newtheorem{definition}[theorem]{Definition}

\theoremstyle{remark}
\newtheorem{remark}[theorem]{Remark}

\numberwithin{equation}{section}

\begin{document}

\title{Free Spaces Over Countable Compact Metric Spaces}

\author{A. DALET}
\address{Laboratoire de Math\'ematiques de Besan\c con, CNRS UMR 6623,\\
Universit\'e de Franche-Comt\'e, 16 Route de Gray, 25030 Besan\c con Cedex, FRANCE.}
\curraddr{}
\email{aude.dalet@univ-fcomte.fr}
\thanks{The first author was partially supported by PHC Barrande 26516YG}


\date{July 22, 2013 and, in revised form, March 11, 2014.}



\begin{abstract}
We prove that the Lipschitz-free space over a countable compact metric space is isometric to a dual space and has the metric approximation property.
\end{abstract}


\newtheorem{Prop}[theorem]{Proposition}
\newtheorem{Coro}[theorem]{Corollary}

\newcommand{\N}{\mathbb{N}}
\newcommand{\Z}{\mathbb{Z}}
\newcommand{\Q}{\mathbb{Q}}
\newcommand{\C}{\mathbb{C}}
\newcommand{\R}{\mathbb{R}}
\newcommand{\F}[1]{\mathcal{F}(#1)}
\newcommand{\norm}[1]{\|#1\|}


\maketitle

\section{Introduction}

Let $(M,d)$ be a pointed metric space, that is to say a metric space equipped with a distinguished origin, denoted $0$. The space $Lip_0(M)$ of Lipschitz functions from $M$ to $\R$ vanishing at $0$ is a Banach space equipped with the Lipschitz norm: 
$$	\norm{f}_L=\sup_{x\neq y}\cfrac{|f(x)-f(y)|}{d(x,y)}.	$$  
Its unit ball is compact with respect to the pointwise topology, then $Lip_0(M)$ is a dual space. In \cite{GK}, its predual is called the Lipschitz free space over $M$, denoted $\F{M}$ and it is the closed linear span of $\left\lbrace \delta_x\ , \ x\in M\right\rbrace $ in $Lip_0(M)^*$. One can prove that the map $\delta : M \rightarrow \F{M}$ is an isometry. For more details on the basic theory of the spaces of Lipschitz functions and their preduals, called Arens-Eells space there, see \cite{W}.

Very little is known about the structure of Lipschitz-free spaces. For instance $\F{\R}$ is isomorphically isometric to $L_1$, but A. Naor and G. Schechtman \cite{NS} proved that $\F{\R^2}$ is not isomorphic to any subspace of $L_1$. The study of the Lipschitz-free space over a Banach space is useful to learn more about the structure of this Banach space. For example G. Godefroy and N. Kalton \cite{GK} proved, using this theory, that if a separable Banach space $X$ isometrically embeds in a Banach space $Y$, then $Y$ contains a linear subspace which is linearly isometric to $X$.

\medskip

We recall that a Banach space $X$ is said to have the approximation property (AP) if for every $\varepsilon >0$ and every compact set $K\subset X$, there is a bounded finite-rank linear operator $T:X\rightarrow X$ such that $\norm{Tx-x}\leq \varepsilon$ for every $x\in K$. If moreover there exists $1\leq\lambda <+\infty$ not depending on $\varepsilon$ or $K$ such that $\norm{T}\leq \lambda$ then $X$ has the $\lambda$-bounded approximation property ($\lambda$-BAP) and $X$ has the bounded approximation property (BAP) if it has the $\lambda$-BAP for some $\lambda$. Finally $X$ has the metric approximation property (MAP) if $\lambda=1$.

\medskip

It is already known that $\F{\R^n}$ has the MAP \cite{GK}, and that if $M$ is a doubling metric space then $\F{M}$ has the BAP \cite{LP}. Moreover E. Perneck\'a and P. H\'ajek \cite{HP} proved that $\F{\ell_1}$ and $\F{\R^n}$ have a Schauder basis.
However, G. Godefroy and N. Ozawa \cite{GO} constructed a compact metric space $K$ such that $\F{K}$ fails the AP. 

\medskip

In the first part of this article we will prove that the Lipschitz-free space over a countable compact metric space $K$ is isometrically isomorphic to the dual space of $lip_0(K)\subset Lip_0(K)$. 
Let $\omega_1$ be the first uncountable ordinal, we will prove, by induction on $\alpha <\omega_1$ such that $K^{(\alpha)}$ is finite, that $\F{K}$ has the MAP. This will rely on a theorem of A. Grothendieck \cite{Gr} asserting that any separable dual having the BAP has the MAP, and a decomposition of the space $K$ due to N. Kalton \cite{K1}.
This provides a negative answer to Question 2 in \cite{GO}, which was originally asked by G. Aubrun to G. Godefroy during a seminar in Lyon about his paper with N. Ozawa.

\section{Duality} 

For any pointed metric space $(M,d)$ we denote by $lip_0(M)$ the subspace of $Lip_0(M)$ defined as follows: $f\in lip_0(M)$ if and only if for every $\varepsilon>0$, there is a $\delta >0$ such that for $x,y\in M$, $d(x,y)<\delta $ implies $|f(x)-f(y)|\leq \varepsilon d(x,y)$.

\medskip

The main result of this section is the following:

\begin{theorem}\label{Thmdual}
If $(K,d)$ is a countable compact metric space, then $\F{K}$ is isometrically isomorphic to a dual space, namely $lip_0(K)^*$.
\end{theorem}

\smallskip

\begin{definition}~\
 \begin{enumerate}
  \item Let $X$ be a Banach space. A subspace $S$ of $X^*$ is called separating if $x^*(x)=0$ for all $x^*\in S$ implies $x=0$.
  \item  For $(M,d)$ a pointed metric space,  $lip_0(M)$ separates points uniformly if there exists a constant $c\geq1$ such that for every $x,y\in M$, some $f\in lip_0(M)$ satisfies $\norm{f}_L\leq c$ and $|f(x)-f(y)|=d(x,y)$. 
 \end{enumerate}
\end{definition}

Mimicking an argument from \cite{G1} we will use a theorem due to Petun$\bar{\i}$n and Pl$\bar{\i}\mathrm{\check{c}}$hko \cite{PP} saying that  if $(X,\norm{\cdot})$ is a separable Banach space and $S$ a closed subspace of $X^*$ contained in  $NA(X)$ (the subset of $X^*$ consisting of all linear forms which attain their norm) and separating points of $X$, then $X$ is isometrically isomorphic to $S^*$. Theorem 3.3.3 in \cite{W} gives the same result but in a less general case.

\medskip

We start with two lemmas taken from \cite{G1}.

\begin{lemma}
 For any $(K,d)$ compact pointed metric space, the space $lip_0(K)$ is a subset of $NA(\F{K})$.
\end{lemma}

\begin{proof}
 We can see $lip_0(K)$ as the subset of $Lip_0(K)$ containing all $f$ such that for every $\varepsilon >0$, the set $K^2_{\varepsilon}:=\left\lbrace (x,y)\in K^2,\ x\neq y,\ |f(x)-f(y)|\geq \varepsilon \  d(x,y)\right\rbrace $ is compact.

 Let $f\in lip_0(K)$, we may assume that $f\neq 0$, then there exists $\varepsilon>0$ such that,
\begin{align*}
 \norm{f}_L&=\sup_{x\neq y} \frac{|f(x)-f(y)|}{d(x,y)}=\sup_{(x, y)\in K_{\varepsilon}^2} \frac{|f(x)-f(y)|}{d(x,y)}=\max_{(x, y)\in K_{\varepsilon}^2} \frac{|f(x)-f(y)|}{d(x,y)}
\end{align*}
Thus there exist $x\neq y$ such that $\norm{f}_L=\frac{|f(x)-f(y)|}{d(x,y)}$ and setting $\gamma =\frac{1}{d(x,y)}(\delta_x-\delta_y)$ we obtain $\gamma  \in \F{K}$ and $\norm{f}_L=|f(\gamma)| $, with $\norm{\gamma}_{\F{K}}=1$ because $\delta$ is an isometry. Then $f$ is norm attaining and $lip_0(K)\subset NA(\F{K})$.
\end{proof}

\begin{lemma}\label{L2.4}For any $(K,d)$ compact pointed metric space, if $lip_0(K)$ separates points uniformly, then it is separating.
\end{lemma}

\begin{proof}

Using Hahn-Banach theorem, one can prove that $lip_0(K)$ is separating if and only if  it is weak$^*$- dense in $Lip_0(K)$.

\medskip

Now assume $lip_0(K)$ separates points uniformly. Then there exists $c\geq 1$ such that for every $F\subset K$, $F$ finite, and every $f\in Lip_0(K)$, we can find $g\in lip_0(K)$, $\norm{g}_L\leq c\norm{f}_L$, such that $f_{|F}=g_{|F}$ (see Lemma 3.2.3 in \cite{W}), and it is easy to deduce that $\overline{lip_0(K)} \ ^{w^*}=Lip_0(K)$.

\end{proof}

These lemmas allow us to reduce the problem. We need to prove that the little Lipschitz space over a countable compact metric space separates points uniformly.

\medskip

For this proof we will use a characterization of countable compact metric spaces with the Cantor-Bendixon  derivation:
for a metric space $(M,d)$ we denote 
\begin{itemize}
\item $M'$ the set of accumulation points of $M$.
\item $M^{(\alpha)}=(M^{(\alpha -1)})'$, for a successor ordinal $\alpha$.
\item $M^{(\alpha)}=\bigcap\limits_{\beta<\alpha}M^{(\beta)}$, for a limit ordinal $\alpha$.
\end{itemize}
A compact metric space $(K,d)$ is countable if and only if there is a countable ordinal $\alpha$ such that $K^{(\alpha)}$ is finite.

\renewcommand{\proofname}{Proof of Theorem 2.1:}
\begin{proof}

Let us prove that $$\exists c\geq1, \ \forall x,y\in K, \ \exists h\in  lip_0(K), \ \norm{h}_L\leq c, \ |h(x)-h(y)|=d(x,y).$$

\medskip

So let $x\neq y \in K$ and set $a=d(x,y)$.
Since $K$ is countable and compact the closed ball $\overline{B}\left(  x, \frac{a}{2}\right)$ of center $x$ and radius $\frac{a}{2}$  is countable and compact and there exists a countable ordinal $\alpha_0$ such that $\overline{B}\left(  x, \frac{a}{2}\right)^{(\alpha_0)}$ is finite and non empty: there exist $k_1\in \N$, $y^1_1, \cdots, y_1^{k_1}\in K$ such that 
$	\overline{B}\left(  x, \frac{a}{2}\right)^{\left(\alpha_0\right)}=\{y^1_1, \cdots, y_1^{k_1}\}$ .
We denote $a^i_1=d(y^i_1,x)$, for $1\leq i\leq k_1$.\  Then we can find $r_1$ and $v_1^1< \cdots < v^{r_1}_1$ such that 
$	\{a^1_1, \cdots, a_1^{k_1}\}=\{v_1^1, \cdots , v^{r_1}_1\}.	$
Now set 
$$
v_1 = \left\{
    \begin{array}{ll}
        a/2, &\!\!\!\!\!\!\!    \textrm{\ if\ } \overline{B}\left(  x, \frac{a}{2}\right)^{(\alpha_0)}=\{x\}\\
        \min\left(\left(\{v^1_1, \ \frac{a}{2}-v_1^{r_1}\}\backslash\{0\}\right)\cup\{v^i_1-v^{i-1}_1 , \  2\leq i \leq r_1\} \right)\!,&\!\!\!\!\!\!\! \textrm{\ otherwise}
    \end{array}
\right.
$$
and define $\varphi_1 : \left[0,+\infty\right[\rightarrow\left[0,+\infty\right[$ by 

$$
\varphi_1(t) = \left\{
    \begin{array}{ll}
        0 & , \  t \in \left[ 0,\frac{v_1}{4}\right[:=V^0_1  \\
        v^i_1 &,  \ t\in \left]v^i_1-\frac{v_1}{4},v^i_1+\frac{v_1}{4} \right[:=V^i_1 , \  1\leq i \leq r_1\\
	\frac{a}{2} &, \  t\in \left]\frac{a}{2}-\frac{v_1}{4},+\infty \right[:=V_1^{r_1+1} 
    \end{array}
\right.
$$

\noindent and $\varphi_1$ is continuous on $[0,+\infty[$ and affine on each interval of $[0,+\infty[ \  \backslash \displaystyle \cup_{i=0}^{r_1+1}V_1^i$.

One can check that the slope of $\varphi_1$ is at most $2$ on each of these intervals, so $\norm{\varphi_1}_L\leq 2$.

\medskip

With $f(\cdot)=d(\ \cdot\ ,x)$ we set $C_1=f^{-1}\left(	[0,+\infty[\backslash \cup_{i=0}^{r_1+1}V_1^i	\right)$. 

If $C_1$ is finite or empty define $h(\cdot)=2\left( \varphi_1\circ d(\ \cdot\ ,x)-\varphi_1(d(0,x))\right)$. It is clear from the definition of $\varphi_1$ that $\norm{h}_L\leq 4$, $|h(x)-h(y)|=d(x,y)$ and $h(0)=0$. Now we set 
$$\delta = \left\{  \begin{array}{ll}
v_1/2, &\textrm{if}\ C_1=\emptyset\\
\cfrac{1}{2}\min \left( \left\{ v_1, \textrm{sep}(C_1)  \right\}\cup \{\textrm{dist}(z, K\backslash C_1), \ z\in  D_1\} \right), &\textrm{otherwise}
\end{array}\right.$$
where $\textrm{sep}(C_1)=\inf\{d(z,t), \ z\neq t, \ z,t\in C_1\}$ and $D_1=f^{-1}\left(	[0,+\infty[\backslash \cup_{i=0}^{r_1+1}\overline{V_1^i}	\right)$. 
Note that $\delta >0$. Indeed $v_1>0$, $C_1$ is finite thus $\textrm{sep}(C_1)>0$, for any $z\in D_1$ $\textrm{dist}(z,K\backslash C_1)>0$ and $D_1$ is finite.

If follows that every $z\neq t\in K$ such that $d(z,t)\leq \delta$ are not in $D_1$ and there exists $i\leq r_1$ such that $z,t \in f^{-1}\left(\overline{V_1^i}\right)$, so the equality $h(z)=h(t)$ holds, i.e. $h\in lip_0(K)$.

\medskip

Assume that $C_1$ is infinite. 
Since $C_1\subset \overline{B}\left(  x, \frac{a}{2}\right)$ we have that for every ordinal $\alpha$, $C_1^{(\alpha)}\subset \overline{B}\left(  x, \frac{a}{2}\right)^{(\alpha)}$. But $C_1\cap \overline{B}\left(  x, \frac{a}{2}\right)^{(\alpha_0)}=\emptyset $ so $C_1^{(\alpha_0)}=\emptyset$. However $C_1$ is compact, thus
there exists $1\leq \alpha_1<\alpha_0$ so that $C_1^{(\alpha_1)}$ is finite and non empty. Then there exist $k_{2}\in \N$ and $  y_{2}^1, \cdots, y_{2}^{k_{2}}\in K$, such that $C_1^{(\alpha_1)}=\lbrace y_{2}^1,\cdots, y_{2}^{k_{2}}\rbrace $.

For $1\leq i\leq k_{2}$, we denote $a^i_{2}=d(y^i_{2},x)$. We can find $r_{2}$ and $v^1_{2}< \cdots < v^{r_{2}}_{2}$ such that 
$$	\{a_{2}^1, \cdots , a_{2}^{k_{2}}\}=\{v^1_{2}, \cdots , v^{r_{2}}_{2}\}.	$$
Now set $$v_{2}=\min\left(\{ v_1 ,\ v_{2}^1\}\cup\{v^i_{2}-v^{i-1}_{2} , \  2\leq i \leq r_{2}\}\right) $$ and define $\varphi_{2} : \left[0,+\infty\right[\rightarrow\left[0,+\infty\right[$ continuous by 

$$
\varphi_{2}(t) = \left\{
    \begin{array}{ll}
        \varphi_1(t) & , \  t \in \displaystyle \bigcup_{i=0}^{r_1+1}V^i_1  \\
        \varphi_1(v^i_{2}) &, \  t\in \left]v^i_{2}-\frac{v_{2}}{2^{3}},v^i_{2}+\frac{v_{2}}{2^{3}} \right[:=V^i_{2} , \ 1\leq i \leq r_{2}\\
    \end{array}
\right.
$$
 and $\varphi_2$ is affine on each interval of 
$[0,+\infty[\backslash\!\! \left(\left(\cup_{i=0}^{r_1+1}V_1^i\right) \cup\left(\cup_{i=1}^{r_2}V_2^i\right)\right).$

The Lipschitz constant of $\varphi_{2}$ equals the maximum between $\norm{\varphi_1}_L$ and new slopes of $\varphi_{2}$. It is easy to check that  $\norm{\varphi_{2}}_L\leq  2\times (1+\frac{1}{3})=\frac{8}{3}$.

Set $C_2=f^{-1}( [\frac{v_1}{4},\frac{a}{2}-\frac{v_1}{4} ] \backslash \left(\left(\cup_{i=1}^{r_1}V_1^i\right) \cup\left(\cup_{i=1}^{r_2}V_2^i\right)\right) )$.

If $C_2$ is finite or empty then setting $h(\cdot)=2\left( \varphi_2\circ d(\ \cdot\ ,x)-\varphi_2(d(0,x))\right)$, we obtain $\norm{h}_L\leq \frac{16}{3} $, $|h(x)-h(y)|=d(x,y)$, $h(0)=0$ and with 
$$0<\delta = \left\{  \begin{array}{ll}
v_2/2, &\textrm{if}\ C_2=\emptyset\\
\cfrac{1}{2}\min \left( \{ v_2  , \textrm{sep}(C_2)\} \cup \{ \textrm{dist}(z,K\backslash C_{2}), \  z\in D_2 \} \right), &\textrm{otherwise}
\end{array}\right.$$

\noindent where $D_2=f^{-1}\left( [\frac{v_1}{4},\frac{a}{2}-\frac{v_1}{4} ] \backslash \left(\left(\cup_{i=1}^{r_1}\overline{V_1^i}\right) \cup\left(\cup_{i=1}^{r_2}\overline{V_2^i}\right)\right) \right)$.

\noindent When $z,t\in K$ are such that $d(z,t)\leq \delta$, then $h(z)=h(t)$, i.e. $h\in lip_0(K)$.

\medskip

If $C_2$ is infinite we proceed inductively in a similar way until we get $C_n$ finite, which eventually happens because we have a decreasing sequence of ordinals.

\medskip

The function $h$ we obtain verifies $h(0)=0$, $|h(y)-h(x)|=d(x,y)$ and 
$$ \norm{h}_L\leq \displaystyle 2\prod_{j=1}^{n}\left(1+\cfrac{1}{2^j-1} \right)\leq \displaystyle 2\prod_{j=1}^{+\infty}\left(1+\cfrac{1}{2^j-1} \right):=c
$$
\noindent where $c$ does not depend on $x$ and $y$.
Moreover, setting 
$$0<\delta = \left\{  \begin{array}{ll}
v_n/2, &\textrm{if}\ C_n=\emptyset\\
\cfrac{1}{2}\left(\min \{ v_n  , \textrm{sep}(C_n) \} \cup \{ \textrm{dist}(z,K\backslash C_{n}), \  z\in D_n \} \right), &\textrm{otherwise}
\end{array}\right.$$
\noindent if $z,t\in K$ are such that $d(z,t)\leq \delta$, then $h(z)=h(t)$, i.e. $h\in lip_0(K)$.

This concludes the proof.

\end{proof}
\renewcommand{\proofname}{Proof:}


\section{Metric Approximation Property}

\begin{theorem}\label{Thm}
Let $(K,d)$ be a countable compact metric space. Then $\mathcal{F}(K)$ has the metric approximation property.
\end{theorem}

Before starting the proof let us recall a construction due to N. Kalton \cite{K1}.

\medskip 

Let $(K,d)$ be an arbitrary pointed metric space and set
\begin{align*}
K_n&=\left\lbrace x\in K, \  d(0,x)\leq 2^n\right\rbrace \textrm{\ and \ } O_n=\left\lbrace x\in K, \  d(0,x)< 2^n\right\rbrace, \ n\in \Z\\
F_N&=K_{N+1}\backslash O_{-N-1}, \ N\in \N.
\end{align*}

Then, for every $n\in \Z$, we can define a linear operator $T_n : \F{K}\rightarrow \F{K}$ by:

$$
T_{n}\delta (x) = \left\{
    \begin{array}{ll}
        0 & , \  x \in K_{n-1}  \\
        \left(\log_2 d(0,x)-(n-1) \right)\delta (x)  &, \  x\in K_n\backslash K_{n-1}\\
	\left( n+1-\log_2d(0,x)\right) \delta (x) 	&, \  x\in K_{n+1}\backslash K_n\\
	0 &,  \ x  \notin K_{n+1} 
    \end{array}
\right.
.$$

If we set for $N\in \N$, $S_N=\displaystyle\sum_{n=-N}^{N}T_n$, then Lemma 4.2 in \cite{K1} gives: 
\begin{lemma}
For every $N\in \N$, we have $\norm{S_N}\leq 72$,  $S_N\left( \F{K}\right)
\subset \F{F_N} $ and
for every $\gamma \in \F{K}$, $\displaystyle\lim_{N\rightarrow+\infty}S_N\gamma=\gamma$.
\end{lemma}

In order to prove Theorem \ref{Thm} we need the following classical lemma. We will give its proof for sake of completeness.

\begin{lemma}\label{Lem3.3}
 If for $\alpha$ countable ordinal there exist $F_1,\cdots,F_n$ clopen subsets of $K^{(\alpha)}$, mutually disjoint, such that $K^{(\alpha)}=F_1\cup\cdots\cup F_n$, then there exist $G_1,\cdots,G_n$ clopen subsets of $K$, mutually disjoint, such that $K=G_1\cup\cdots\cup G_n$ and $G_i^{(\alpha)}=F_i$.
\end{lemma}
\renewcommand{\proofname}{Proof:}
 \begin{proof}
We proceed by induction on $\alpha <\omega_1$ such that $K^{(\alpha)}=F_1\cup\cdots\cup F_n$, for all $1\leq i\neq j \leq n$, $F_i$ is clopen in $K^{(\alpha)}$ and $F_i\cap F_j=\emptyset $. 
 
 \medskip
 
\noindent The result is clear for $\alpha=0$.
 
 \medskip
 
 \noindent  Assume that the result is true for $\alpha<\omega_1$ and suppose that $\left\{F_i\right\}_{1\leq i\leq n}$ is a clopen partition of $K^{(\alpha+1)}$.

Each $F_i$ is closed in $K^{(\alpha)}$ which is compact, then we can find $O_i$ open subset of $K^{(\alpha)}$ such that $F_i\subset O_i$, $O_i'=F_i$, and  $O_i\cap O_j=\emptyset$, for $i\neq j$. Set $O=K^{(\alpha)}\backslash \cup_{i=1}^{n}O_i$, $U_1=O_1\cup O$ and $U_i=O_i$, for $2\leq i \leq n$. Then $K^{(\alpha)}=\cup_{i=1}^{n}U_i$, $U_i'=F_i$, and every $U_i$ is clopen in $K^{(\alpha)}$. 
Indeed we defined $O_i, 2\leq i\leq n$, as open subsets of $K^{(\alpha)}$ so $U_i$ is open in  $K^{(\alpha)}$. Moreover points in $O$ are isolated points of $K^{(\alpha)}$ thus $O$ and then  $U_1$ are open in $K^{(\alpha)}$. Finally $K^{(\alpha)}=\cup_{i=1}^{n}U_i$ then every $U_i$'s are closed.

\noindent We can apply the induction hypothesis to find $G_1,\cdots,G_n$ clopen subsets of $K$, mutually disjoints, such that $K=G_1\cup\cdots\cup G_n$ and $G_i^{(\alpha)}=U_i$, that is $G_i^{(\alpha+1)}=F_i$, $1\leq i \leq n$ .
 
 \medskip 
 
\noindent Finally we assume $\alpha$ is a limit ordinal and $K^{(\alpha)}=F_1\cup\cdots\cup F_n$, disjoint union  of clopen sets in $K^{(\alpha)}$. There exist $O_1,\cdots,O_n$ open subsets of $K$ such that $F_i\subset O_i$, $O_i^{(\alpha)}=F_i$ and $O_i\cap O_j=\emptyset$ for $i\neq j$. 

\noindent Set $F=K\backslash \cup_{i=1}^nO_i$, then $\bigcap_{\beta <\alpha}F\cap K^{(\beta)}=F\cap K^{(\alpha)}=\emptyset$. But $F$ is compact, then there exists $\beta < \alpha$ such that $F\cap K^{(\beta)}=\emptyset $, that is to say $K^{(\beta)}\subset \cup_{i=1}^nO_i$. Finally $K^{(\beta)}$ is the disjoint union of $O_i\cap K^{(\beta)}, \ 1\leq i \leq n$, clopen sets in $K^{(\beta)}$, so we can use the induction hypothesis to write $K=G_1\cup\cdots\cup G_n$, $G_i$ mutually disjoint and clopen in $K$ and $G_i^{(\beta)}=O_i\cap K^{(\beta)}=O_i^{(\beta)}$. Moreover we have $\beta<\alpha$ thus $G_i^{(\alpha)}=\bigcap_{\gamma<\alpha}G_i^{(\gamma)}=\bigcap_{\gamma<\alpha}O_i^{(\gamma)}=O_i^{(\alpha)}=F_i$.
 \end{proof}

\renewcommand{\proofname}{Proof of Theorem 3.1:}
\begin{proof}

We proceed by induction on $\alpha<\omega_1$ such that $K^{(\alpha)}$ is finite.

\begin{itemize}
\item If $K$ is finite then $\F{K}$ is finite dimensional, so has trivially the MAP and the property is true for $\alpha =0$.
\item Let $\alpha $ be a countable ordinal and assume that for every $\beta < \alpha$, if $(K,d)$ is a compact metric space so that $K^{(\beta)}$ is finite, then $\F{K}$ has the MAP.  

Now let $(K,d)$ be a compact metric space such that $K^{(\alpha)}$ is finite.

First $\F{K}$ is linearly isometric to $lip_0(K)^*$ and a theorem of A. Grothendieck \cite{Gr} (see also Theorem 1.e.15 in \cite{LT}) asserts that a separable Banach space which is isometric to a dual space and which has the AP has the MAP,  so it is enough to prove that $\F{K}$ has the BAP.

Secondly, if $K$ is such that $K^{(\alpha)}=\left\lbrace a_1,\cdots,a_n\right\rbrace $, singletons $\{a_i\}$'s are clopen in $K^{(\alpha)}$ and Lemma \ref{Lem3.3} gives $G_1,\cdots,G_n$
mutually disjoint clopen subsets of $K$ such that $\forall i \leq n$, $G_i^{(\alpha)}=\left\lbrace a_i \right\rbrace $ and $K=G_1\cup\cdots\cup G_n$. Moreover $\F{K}$ is isomorphic to $\displaystyle\left( \oplus_{i=1}^n\F{K_i}\right)_{\ell_1} $, where $K_i=G_i\cup\left\lbrace 0\right\rbrace $, $1\leq i \leq n$.

 Indeed if 
$\displaystyle a=\min_{i\neq j}\textrm{dist}(G_i,G_j) $, where $$\textrm{dist}(G_i,G_j)=\inf \left\lbrace d(x,y)\ ,\ x\in G_i\ , \ y\in G_j \right\rbrace,$$ by compactness we have $a>0$. Then  the operator 

$$
\Phi:
    \begin{array}{rcl}
        Lip_0(K) & \rightarrow & \left( \oplus_{i=1}^{n}Lip_0(K_i)\right)_{\infty}   \\
        f  &\mapsto &\left( f_{|K_i}\right)_{i=1}^{n}  \\
    \end{array}
$$

\noindent is onto, linear, weak$^{*}$-continuous and for $f\in Lip_0(K)$, we have $$\cfrac{a}{2 \ \textrm{diam}(K)}\norm{f}_L\leq \norm{\Phi (f)}_{\infty}\leq \norm{f}_L.$$
Hence $\F{K}$ is isomorphic to $\displaystyle\left( \oplus_{i=1}^n\F{K_i}\right)_{\ell_1} $.

\medskip

The BAP is stable with respect to finite $\ell_1-$sums and isomorphisms, then it is enough to prove that for any $ i\in \{1, \cdots, n\}$, $\F{K_i}$ has the BAP. In other words we need to prove that when $K^{(\alpha)}$ is a singleton then $\F{K}$ has the BAP.

\medskip 

Suppose as we may that $K^{(\alpha)}=\left\lbrace 0\right\rbrace $. Using the construction due to Kalton \cite{K1} we have a sequence of linear operators $S_N: \F{K}\rightarrow \F{F_N}$, $\norm{S_N}\leq 72$ and for every $\gamma \in \F{K}$, $\displaystyle \lim_{N\rightarrow +\infty}S_N\gamma =\gamma$.

Moreover, for every $N\in \N$, there exists $\beta <\alpha$ such that $F_N^{(\beta)} $ is finite and then $\F{F_N}$ has the MAP: 

since $\F{F_N}$ is separable,
 for every $N\in \N$, there exists a sequence of finite-rank linear operators $R_p^N: \F{F_N}\rightarrow\F{F_N}$  so that for every $ \gamma \in \F{F_N}$, $\displaystyle \lim_{p\rightarrow+\infty}R_p^N\gamma =\gamma$ and $\norm{R_p^N}\leq 1$ for every $p\in \N$ (\cite{Pe}, see also Theorem 1.e.13 in \cite{LT}).

Setting $Q_{N,p}=R_p^N\circ S_N$ we deduce that the range of $Q_{N,p}$ is finite dimensional as the range of $R_p^N$, $\norm{Q_{N,p}}\leq \norm{R_p^N}\norm{S_N}\leq 72$ and for every $\gamma \in \F{K}$, 
$$\lim_{N\rightarrow +\infty}\lim_{p\rightarrow +\infty}R_p^NS_N\gamma=\lim_{N\rightarrow +\infty}S_N\gamma=\gamma.$$ 
Thus $\F{K}$ has the $72$-BAP and this concludes the proof.
\end{itemize}
\end{proof}

\renewcommand{\proofname}{Proof:}

\section{Application to quotient spaces}

For a pointed metric space $(M,d)$ and $A$ a closed subset of $M$ containing $0$ we can define the quotient $M/A$ as the space $(M\backslash A) \cup \left\lbrace 0\right\rbrace $ with the metric given by: 

$$
d_{M/A}(x,y)=\left\{
    \begin{array}{ll}
        \mathrm{dist}(x,A) & , \ y=0  \\
        \min \left\lbrace d(x,y),\ \mathrm{dist}(x,A)+\mathrm{dist}(y,A) \right\rbrace  &, \  x,y\neq 0\\
    \end{array}
\right.
.$$

\begin{Coro}\label{Coro4.1}
Let $(K,d)$ be a compact metric space which is not perfect (i.e. $K'\neq K$). Then for every countable ordinal $\alpha\geq 1$, the space $\F{K/K^{(\alpha)}}$ has the MAP.
\end{Coro}

\begin{proof}
Remark that for every compact metric space $(K,d)$ and every countable ordinal $\alpha \geq 1$, the quotient space $K/K^{(\alpha)}$ is compact and countable because $\left(K/K^{(\alpha)}\right)^{(\alpha)}$ is empty or $\{0\}$. Then this result is a consequence of Theorem \ref{Thm}.
\end{proof}

\begin{remark}
\begin{enumerate}
\item If $K$ is perfect, then $\F{K/K^{(\alpha)}}=\{0\}$.
\item Otherwise $\F{K}/\F{K^{(\alpha)}}$ is linearly isometric to $\F{K/K^{(\alpha)}}$ (We write $\F{K}/\F{K^{(\alpha)}}\equiv\F{K/K^{(\alpha)}}$).

Indeed we can assume that $0\in K^{(\alpha)}$. Then 
\begin{align*}
&\{f\in Lip_0(K)\ ; \ \forall x,y\in K^{(\alpha)}, f(x)=f(y)\}\\&=\{f\in Lip_0(K)\ ; \ \forall x\in K^{(\alpha)}, f(x)=0\}.
\end{align*}
  And since $\F{K^{(\alpha)}}=\overline{\textrm{vect}}\ \{\delta_x, \ x\in K^{(\alpha)}\}$, we have $$\{f\in Lip_0(K)\ ; \ \forall x\in K^{(\alpha)}, f(x)=0\}=\F{K^{(\alpha)}}^{\perp},$$ which is isometric to $\left(\F{K}/\F{K^{(\alpha)}}\right)^*$. To sum up $$\{f\in Lip_0(K)\ ; \ \forall x,y\in K^{(\alpha)}, f(x)=f(y)\}\equiv\left(\F{K}/\F{K^{(\alpha)}}\right)^*.$$

From Propositions 1.4.3 and 1.4.4 in \cite{W},
 there exists an isometry $\Phi$ from $\{f\in Lip_0(K)\ ; \ \forall x,y\in K^{(\alpha)}, f(x)=f(y)\}$ onto $Lip_0\left(K/K^{(\alpha)}\right)$. Moreover $Lip_0(K/K^{(\alpha)})$ is linearly isometric to $\F{K/K^{(\alpha)}}^*$, so  the space $\left(\F{K}/\F{K^{(\alpha)}}\right)^*$ is isomorphically isometric to $\F{K/K^{(\alpha)}}^*$. One can easily check that $\Phi$ is weak$^*$-continuous, and finally $\F{K}/\F{K^{(\alpha)}}$ is linearly isometric to $\F{K/K^{(\alpha)}}$.
 \end{enumerate}
\end{remark}

To finish this paper we will use Corollary \ref{Coro4.1} and the previous remark to prove the following:
in order to obtain that every countable compact metric space has the BAP it
 is not possible to use the three-space property due to G. Godefroy and P.D. Saphar \cite{GS}, asserting: 

If $M$ is a closed subspace of a Banach space $X$ so that $M^{\bot}$ is complemented in $X^*$ and $X/M$ has the BAP, then $X$ has the BAP if and only if $M$ has the BAP.
 
Indeed we can construct a compact metric space $K$ so that $K^{(2)}=\left\lbrace 0\right\rbrace $, in particular $\F{K}$, $\F{K'}$ and $\F{K}/\F{K'}$ have the MAP, but $\F{K'}^{\bot}$ is not complemented in $Lip_0(K)$.

To construct this space we need a proposition similar to Proposition 7 in \cite{GO}: 

\begin{Prop}\label{PropProj}
 For any $\lambda>0$, there exist a finite metric space $H_{\lambda}$ and a subset $G_{\lambda}$ of $H_{\lambda}$ such that if $P:Lip_0(H_{\lambda})\rightarrow \F{G_{\lambda}}^{\bot}$ is a bounded linear projection, then $\norm{P}\geq \lambda$.
\end{Prop}

\begin{proof}
 Assume that for some $\lambda_0>0$ and for all pairs $(G,H)$ of finite metric spaces with $G\subset H$ we can construct $P:Lip_0(H)\rightarrow \F{G}^{\bot}$ linear projection with norm bounded by $\lambda_0$.

Let $K$ be the compact metric space such that $\F{K}$ fails AP appearing in Corollary 5 of \cite{GO}. There exists $(G_n)_{n\in \N}$ an increasing sequence of finite subsets of $K$ such that $\displaystyle\overline{\bigcup_{n\in \N}G_n}=K$.

Then for every $n\in \N$ and every $k\geq n$, there exists $P_n^k:Lip_0(G_k)\rightarrow \F{G_n}^{\bot}$ a linear projection of norm less than $\lambda_0$, where $\F{G_n}^{\bot}\subset Lip_0(G_k)$.  

\medskip 

Fix $n\in \N$. For $k\in \N$, let $E_k: Lip_0(G_k)\rightarrow Lip_0(K)$ be the non linear extension operator which preserves the Lipschitz constant  given by the inf-convolution formula: 
$$\forall f\in Lip_0(K), \  \forall x\in K, \ E_kf(x)=\inf_{y\in G_k}\lbrace f(y)+\norm{f}_Ld(x,y) \rbrace.$$

For $f\in Lip_0(K)$, we set
$$
\widetilde{P_n^k}(f) = \left\{
    \begin{array}{ll}
       E_kP_n^k\left( f_{|G_k}\right) & , k\geq n  \\
       0 &, k<n
    \end{array}
\right.
.$$

Then $\norm{\widetilde{P_n^k}(f)}_L\leq \lambda_0\norm{f}_L$, for every $f\in Lip_0( K)$.

If $\mathcal{U}$ is a non trivial ultrafilter on $\N$, for every $f\in Lip_0\left( K\right) $ we can define $P_nf$ as the pointwise limit of $\widetilde{P_n^k}(f)$ with respect to $k\in \mathcal{U}$. Then $P_n$ is a linear projection onto $\F{G_n}^{\bot}\subset Lip_0(K)$ because $P_n^k$ is a projection onto $\F{G_n}^{\bot}\subset Lip_0(G_k)$. Moreover $\norm{P_nf}_L\leq \lambda_0 \norm{f}_L$ and $P_nf$ pointwise converges to $0$ for any $f\in Lip_0(K)$.

Set $Q_n=Id_{Lip_0(K)}-P_n: Lip_0(K)\rightarrow Lip_0(K)$. Then $Q_n$ is a continuous linear projection of finite rank and $\mathrm{Ker}\ Q_n =\F{G_n}^{\bot}$ is weak$^*$-closed. Therefore $Q_n$ is weak$^*$-continuous. Moreover  $\norm{Q_n}\leq 1+\lambda_0$ and for every $f\in Lip_0(K)$, $Q_nf$ converges pointwise to $f$.

Using Theorem 2 in \cite{BM} we deduce that $\F{K}$ has the $(1+\lambda_0)$-BAP, contradicting our assumption on $K$.

\end{proof}

Thanks to that proposition we will construct a compact metric space $K$ such that $K^{(2)}=\left\lbrace 0\right\rbrace $ and $\F{K'}^{\bot}$ is not complemented in $Lip_0(K)$:

\medskip 

For every $n\in \N$ there exist $A_n\subset B_n$ finite such that for every continuous linear projection $P_n:Lip_0(B_n)\rightarrow \F{A_n}^{\bot}$, we have $\norm{P_n}\geq n$.
 
Set $\alpha_n=\min\left\lbrace d(x,y)\ , \  x\neq y \in B_n\right\rbrace >0$. If we see $B_n$ as a subspace of $\ell_{\infty}^{m_n}$, with $m_n$ the cardinality of $B_n$, we can find for every $a\in A_n$, $L_n^a$ a sequence converging to $a$ such that $L_n^a\subset B\left( a,\frac{\alpha_n}{2}\right) $.

Define $K_n=\displaystyle\left(\bigcup_{a\in A_n}L_n^a \right)\cup B_n $, we obtain $A_n\subset B_n\subset K_n$ and $K_n'=A_n$.  We can assume that the diameter of $K_n$ is less than $8^{-n}$.

Finally we define $K:=\left( \displaystyle \bigcup_{n\in \N}\left\lbrace n \right\rbrace \times K_n \right)\cup \left\lbrace 0 \right\rbrace $ equipped with the distance: 

$$ \begin{array}{rl}
d(0,(n,x))&=2^{-n}\\

d((n,x),(m,y)) &= \left\{
    \begin{array}{ll}
        d_{K_n}(x,y) & ,\  n=m  \\
        |2^{-n}-2^{-m}|&, \ n\neq m\\
    \end{array}
\right.    \end{array}
.$$

Then $K^{(2)}=\left\lbrace 0\right\rbrace$. 

\medskip 

Now assume that there exists $P:Lip_0(K)\rightarrow \F{K'}^{\bot}$ a continuous linear projection. Let $E_n:Lip_0(B_n)\rightarrow Lip_0(K_n)$, $F_n : Lip_0(B_n)\rightarrow Lip_0(K)$ and  $R_n:Lip_0(K)\rightarrow Lip_0(B_n)$  be defined as follows:

$$\begin{array}{lrl}
 \forall f\in Lip_0(B_n), \ &
(E_nf)(x) &= \left\{
    \begin{array}{ll}
        f(x) & , \ x\in B_n   \\
        f(a) &, \  x \in L_n^a
    \end{array}
\right.\\
\\
 \forall f\in Lip_0(B_n), \ & (F_nf)(m,x)&=\left\{
    \begin{array}{ll}
       (E_nf)(x) & , \  n=m  \\
        0&, \ n\neq m\\
    \end{array}
\right.\\
\\
\forall f\in Lip_0(K), \ &\  R_nf &= f_{|\left\lbrace n\right\rbrace \times B_n} 
\end{array} .$$

\smallskip

 We set $P_n:=R_n\circ P\circ F_n : Lip_0(B_n)\rightarrow \F{A_n}^{\bot}$ and we have that $P_n$ is a continuous linear projection. From Proposition \ref{PropProj} we deduce that $\norm{P_n}\geq n $. Moreover our choice of $\alpha_n$ implies that $\norm{F_n}\leq 3$, then finally $\norm{P}\geq n/3$ holds for every $n\in \N$. Therefore $P$ is unbounded and $\F{K'}^{\bot}$ is not complemented in $Lip_0(K)$.

\medskip

\noindent\textbf{Acknowledgement.} The author would like to thank Gilles Godefroy for fruitful discussions and Gilles Lancien for useful conversations and comments.

\bibliographystyle{amsplain}

\end{document}